\documentclass[letterpaper, 10 pt, conference]{ieeeconf}  

\IEEEoverridecommandlockouts                              
\usepackage{amsmath,amssymb,enumerate,subfigure,multirow,hhline,color}
\usepackage{xcolor}
\usepackage{hyperref}
\usepackage{graphicx}
\usepackage{graphics}
\usepackage[sort]{cite}
\usepackage{algorithm}
\usepackage{algpseudocode}

\usepackage{epsfig,psfrag}
\usepackage{tabularx}
\usepackage{tabulary}

\usepackage[sort]{cite}

\usepackage{hyperref}
\hypersetup{breaklinks=true,colorlinks=true,linkcolor=blue,citecolor=blue}

\usepackage{tcolorbox}

\usepackage[noabbrev,capitalise,nameinlink]{cleveref}

\DeclareMathOperator{\argmin}{arg\,min}

\newtheorem{example}{Example}
\newtheorem{assumption}{Assumption}
\newtheorem{remark}{Remark}

\newtheorem{problem}{Problem}
\newtheorem{lemma}{Lemma}

\newtheorem{proposition}{Proposition}

\allowdisplaybreaks

\title{\LARGE \bf Primal-Dual Algorithm for Distributed Reinforcement Learning: Distributed GTD2}

\title{Primal-Dual Algorithm for Distributed Reinforcement Learning: Distributed GTD}

\author{Donghwan Lee, Hyungjin Yoon, and Naira Hovakimyan
\thanks{This work has been supported in part by the National
Science Foundation through the National Robotics Initiative
grant number 1528036, EAGER grant number 1548409 and  AFOSR  grant number FA9550-15-1-0518.}
\thanks{D. Lee is with Coordinated Science Laboratory (CSL),
University of Illinois, Urbana-Champaign, IL 61801, USA {\tt\small
donghwan@illinois.edu}.}
\thanks{H. Yoon, and N. Hovakimyan are with the Department of Mechanical Science and Engineering,
University of Illinois, Urbana-Champaign, IL 61801, USA
{{\tt\small hyoon33@illinois.edu}, {\tt\small
nhovakim@illinois.edu}.}}}

\begin{document}

\maketitle \thispagestyle{empty} \pagestyle{empty}

\begin{abstract}
The goal of this paper is to study a distributed version of the
gradient temporal-difference (GTD) learning algorithm for
multi-agent Markov decision processes (MDPs). The
temporal-difference (TD) learning is a reinforcement learning (RL)
algorithm which learns an infinite horizon discounted cost
function (or value function) for a given fixed policy without the
model knowledge. In the distributed RL case each agent receives
local reward through a local processing. Information exchange over
sparse communication network allows the agents to learn the global
value function corresponding to a global reward, which is a sum of
local rewards. In this paper, the problem is converted into a
constrained convex optimization problem with a consensus
constraint. Then, we propose a primal-dual distributed GTD
algorithm and prove that it almost surely converges to a set of
stationary points of the optimization problem.
\end{abstract}

\section{Introduction}

 The goal of this paper is to study a distributed
version of the gradient temporal-difference (GTD) learning
algorithm, originally presented in
\cite{sutton2009convergent,sutton2009fast}, for multi-agent Markov
decision processes (MDPs). There are $N$ agents $i \in\{1,\ldots,N
\}=:{\cal V}$, which do not know the statistics of the state
transitions and rewards. Each agent $i$ receives local reward
following a given fixed local policy $\pi_i$. However, it will be
able to learn the global infinite horizon discounted cost function
(or value function) corresponding to the reward that is a sum of
local rewards through information exchange over a sparse
communication network. This paper only focuses on the value
evaluation problem with fixed local policies $\pi_i$. However, the
proposed approach can be extended to actor-critic algorithms,
which have multi-agent cooperative control design applications.

A distributed Q-learning (QD-learning) was studied
in~\cite{kar2013cal}. The focus of~\cite{kar2013cal} is to learn
an optimal Q-factor~\cite{bertsekas1996neuro} for a global reward
expressed as a sum of local rewards, while each agent is only
aware of its local reward. This work therefore addresses the
multi-agent optimal policy design problem. If each agent has
access to partial states and actions, then the transition model of
each agent becomes non-stationary. This is because the state
transition model of each agent depends on the other agents'
policies. In~\cite{kar2013cal}, the authors assumed that each
agent observes the global state and action; therefore, this
non-stationary problem does not occur in Q-learning settings.
Distributed actor-critic algorithms were explored
in~\cite{pennesi2010distributed} with a similar setting. Each
agent acquires local observations and rewards, but it tries to
learn an optimal policy that maximizes the long-term average of
total reward which is a sum of local rewards. It was assumed that
each agent's state-action does not change the other agents'
transition models. In a more recent work~\cite{zhang2018fully},
consensus-based actor-critic algorithms were studied, where the
authors assumed that the transition model depends on the joint
action-states of all agents, and that each agent can observe the
entire combination of action-states.

In~\cite{macua2015distributed}, a distributed policy evaluation
was studied with the GTD
from~\cite{sutton2009convergent,sutton2009fast} combined with
consensus steps. The study focused on the scenario that there
exists only one global reward, each agnet behaves according to
their own behavior policy $\pi_i$, and the agents cooperate to
learn the value function of the target policy $\pi$; thereby, it
is a multi-agent off-policy learning scheme. It was also assumed
that each agent can only explore a small subset of the MDP states.
A consensus-based GTD was also addressed
in~\cite{stankovic2016multi}. The authors considered a problem
similar to~\cite{macua2015distributed}, and the weak convergence
of the algorithm  was proved. In~\cite{mathkar2017distributed}, a
gossip-based distributed temporal difference
(TD~\cite{bertsekas1996neuro}) learning was investigated. Compared
to the previous work, the main difference
in~\cite{mathkar2017distributed} is that all agents know the
global reward, but they have different linear function
approximation architectures with different features and parameters
of different dimensions. Agents cooperate to find a value function
with a linear function approximation consisting of aggregated
features of all agents to reduce computational costs. Lastly, the
papers~\cite{ram2010distributed,bianchi2013convergence} addressed
distributed consensus-based stochastic gradient optimization
algorithms for general convex and non-convex objective functions,
respectively. Whenever the learning task can be expressed as a
minimization of an objective function, e.g.,
GTD~\cite{sutton2009convergent,sutton2009fast} or the residual
method~\cite{baird1995residual}, algorithms
in~\cite{ram2010distributed,bianchi2013convergence} can be
applied. Besides,~\cite{tutunov2016exact} studied a distributed
Newton method for policy gradient methods.

The main contribution of this paper is the development of a new
class of distributed GTD algorithm based on primal-dual iterations
as compared to the original one in~\cite{sutton2009fast}. The most
relevant previous studies which addressed the same problem setting
are~\cite{pennesi2010distributed,zhang2018fully}. Even
though~\cite{pennesi2010distributed,zhang2018fully} studied
actor-critic algorithms, if the actor updates are ignored with
fixed policies, then they can deal with the same problem as ours.
The main difference compared to the previous result is that the
proposed algorithm incorporates the consensus task into an
equality constraint, while those
in~\cite{pennesi2010distributed,zhang2018fully} use the averaging
consensus steps explicitly. Therefore, our algorithm views the
problem as a constrained optimization, and solves it using a
primal-dual saddle point algorithm. The proposed method is mainly
motivated by~\cite{macua2015distributed}, where the GTD was
interpreted as a primal-dual algorithm using Lagrangian duality
theory. The proposed algorithm was also motivated by the
continuous-time consensus optimization algorithm
from~\cite{wang2010control,wang2011control,gharesifard2014distributed},
where the consensus equality constraint was introduced. The recent
primal-dual reinforcement learning algorithm
from~\cite{chen2016stochastic} also inspired the development in
this paper. We also note a primal-dual variant of the GTD
in~\cite{mahadevan2014proximal} with proximal operator approaches.

One of the benefits of the proposed scheme is that the consensus
and learning tasks are unified into a single ODE. Therefore, the
convergence can be proved solely based on the ODE
methods~\cite{borkar2000ode,bhatnagar2012stochastic,kushner2003stochastic},
and the proof is relevantly simpler. The second possible advantage
is that  the proposed algorithm is a stochastic primal-dual method
for solving saddle point problems, and hence some analysis tools
from optimization perspectives, such as~\cite{chen2016stochastic},
can be applied (for instance, the convergence speed and complexity
of the algorithm), and this agenda is briefly discussed at the end
of the paper. Full extension in this direction will appear in an
extended version of this paper. The third benefit of the approach
is that the method can be directly extended to the case when the
communication network is stochastic. In addition, the proposed
method can be generalized to an actor-critic algorithm and
off-policy learning. In this paper, we will focus on a convergence
analysis based on the ODE
approach~\cite{borkar2000ode,bhatnagar2012stochastic,kushner2003stochastic}.
Several open questions remain. For example, it is not clear if
there exists a theoretical guarantee that the proposed algorithm
improves previous consensus
algorithms~\cite{stankovic2016multi,pennesi2010distributed,zhang2018fully}.
Brief discussions are included in the example section.

\section{Preliminaries}
\subsection{Notation}
The adopted notation is as follows: ${\mathbb R}^n $:
$n$-dimensional Euclidean space; ${\mathbb R}^{n \times m}$: set
of all $n \times m$ real matrices; $A^T$: transpose of matrix $A$;
$I_n$: $n \times n$ identity matrix; $I$: identity matrix with an
appropriate dimension; $\|\cdot \|$: standard Euclidean norm; for
any positive-definite $D$, $\|x\|_D:=\sqrt{x^T Dx}$; for a set
${\mathcal S}$, $|{\mathcal S}|$ denotes the cardinality of the
set; ${\mathbb E}[\cdot]$: expectation operator; ${\mathbb
P}[\cdot]$: probability of an event; for any vector $x$, $[x]_i$
is its $i$-th element; for any matrix $P$, $[P]_{ij}$ indicates
its element in $i$-th row and $j$-th column; if ${\bf z}$ is a
discrete random variable which has $n$ values and $\mu \in
{\mathbb R}^n$ is a stochastic vector, then ${\bf z} \sim \mu$
stands for ${\mathbb P}[{\bf z} = i] = [\mu]_i$ for all $i \in
\{1,\ldots,n \}$; ${\bf 1}$ denotes a vector with all entries
equal to one; ${\rm dist}({\cal S},x)$: standard Euclidean
distance of a vector $x$ from a set ${\cal S}$, i.e., ${\rm
dist}({\cal S},x):=\inf_{y\in {\cal S}} \|x-y\|$; for a convex
closed set $\cal S$, $\Gamma_{\cal S}(x):=\argmin_{y\in {\cal S}}
\|x-y\|$.

\subsection{Graph theory}
An undirected graph with the node set ${\cal V}$ and the edge set
${\cal E} \subseteq {\cal V} \times {\cal V}$ is denoted by ${\cal
G}=({\cal E},{\cal V})$. We define the neighbor set of node $i$ as
${\cal N}_i := \{ j\in {\cal V}:(i,j)\in {\cal E}\}$. The
adjacency matrix of $\cal G$ is defined as a matrix $W$ with
$[W]_{ij} = 1$, if and only if $(i,j) \in {\cal E}$. If $\cal G$
is undirected, then $W=W^T$. A graph is connected, if there is a
path between any pair of vertices. The graph Laplacian is $L = H -
W$, where $H$ is diagonal with $[H]_{ii} = |{\cal N}_i|$. If the
graph is undirected, then $L$ is symmetric positive semi-definite.
It holds that $L {\bf 1}=0$. We put the following assumption on
the graph ${\cal G}$.
\begin{assumption}\label{assumption:connected}
${\cal G}$ is connected.
\end{assumption}
Under~\cref{assumption:connected}, 0 is a simple eigenvalue of
$L$.

\section{Reinforcement learning overview}\label{sec:RL-overview}
We briefly review basic RL algorithm
from~\cite{sutton1998reinforcement} with linear function
approximation for the single agent case. A Markov decision process
is characterized by a quadruple ${\cal M}: = ({\cal S},{\cal
A},P,r,\gamma)$, where ${\cal S}$ is a finite state space
(observations in general), $\cal A$ is a finite action space,
$P(s,a,s'):={\mathbb P}[s'|s,a]$ is a tensor that represents the
unknown state transition probability from state $s$ to $s'$ given
action $a$, $r:{\cal S} \times {\cal A} \to {\mathbb R}$ is the
reward function, and $\gamma \in (0,1)$ is the discount factor.
The stochastic policy is a mapping $\pi:{\cal S} \times {\cal A}
\to [0,1]$ representing the probability $\pi(s,a)={\mathbb
P}[a|s]$, $r^\pi(s):{\cal S} \to {\mathbb R}$ is defined as
$r^\pi(s):= {\mathbb E}_{a \sim \pi (s)}[r(s,a)]$, $P^\pi$ denotes
the transition matrix whose $(s,s')$ entry is ${\mathbb P}[s'|s] =
\sum_{a \in {\cal A}}{{\mathbb P}[s'|s,a]\pi (s,a)}$, and $d:{\cal
S} \to {\mathbb R}$ denotes the stationary distribution of the
observation $s\in {\cal S}$. The infinite-horizon discounted value
function with policy $\pi$ and reward $r$ is
\begin{align*}
&J^\pi(s):={\mathbb E}_{\pi,P} \left[ \left. \sum_{k=0}^\infty
{\gamma^{k-1} r^\pi(s_k)} \right|s_0=s \right],
\end{align*}
where ${\mathbb E}_{\pi,P}$ implies the expectation taken with
respect to the state-actor trajectories following the state
transition $P$ and policy $\pi$. Given pre-selected basis (or
feature) functions $\phi_1,\ldots,\phi_q:{\cal S}\to {\mathbb R}$,
$\Phi \in {\mathbb R}^{|{\cal S}| \times q}$ is defined as a full
column rank matrix whose $i$-th row vector is $\begin{bmatrix}
\phi_1(i) &\cdots & \phi_q(i) \end{bmatrix}$. The goal of RL with
the linear function approximation is to find the weight vector $w$
such that $J_{w}=\Phi w$ approximates $J^{\pi}$. This is typically
done by minimizing the {\em mean-square Bellman error} loss
function~\cite{sutton2009fast}
\begin{align}
&\mathop{\min}_{w \in {\mathbb R}^q} {\rm MSBE}(w):=\frac{1}{2}
\left\| r^{\pi}+\gamma P^\pi \Phi w-\Phi w
\right\|_{D}^2,\label{eq:loss-function1}
\end{align}
where $D$ is a symmetric positive-definite matrix. For online
learning, we assume that $D$ is a diagonal matrix with positive
diagonal elements $d(s),s\in {\cal S}$. The residual
method~\cite{baird1995residual} applies the gradient descent type
approach $w_{k+1}=w_{k}-\alpha_k \nabla_w {\rm MSBE}(w)(w)$, where
$\nabla_w {\rm MSBE}(w)=(\gamma P^{\pi} \Phi -
\Phi)^T(r^{\pi}+\gamma P^{\pi} \Phi w -\Phi w)$. In the model-free
learning, the gradient is replaced with a sample-based stochastic
estimate. A drawback of the residual method is that the next
observation $s'$ should be sampled twice to obtain an unbiased
gradient estimate. In the
TD~learning~\cite{sutton1998reinforcement,bertsekas1996neuro} with
a linear function approximation, the problem is resolved by
ignoring the first $\gamma P^{\pi} \Phi$ in the gradient~$\nabla_w
{\rm MSBE}(w)$: $\nabla_w {\rm MSBE}(w)\cong (-\Phi)^T
D(r^{\pi}+\gamma P^{\pi}\Phi w-\Phi w)$. If the linear function
approximation is used, then this algorithm converges to an optimal
solution of~\eqref{eq:loss-function1}. Compared to the residual
method, the double sampling issue does not occur. In the above two
methods, the fixed point problem $r^{\pi}+\gamma P^{\pi}\Phi
w=\Phi w$ may not have a solution in general because the left-hand
side need not lie in the range space of $\Phi$. To address this
problem, the GTD in~\cite{sutton2009fast} solves instead the
minimization of the {\em mean-square projected Bellman error} loss
function
\begin{align}
&\mathop{\min}_{w\in {\mathbb R}^q} {\rm MSPBE}(w):=
\frac{1}{2}\left\| \Pi (r^{\pi}+\gamma P^{\pi} \Phi w-\Phi w)
\right\|_D^2,\label{eq:GDD(0)-loss}
\end{align}
where $\Pi$ is the projection onto the range space of $\Phi$,
denoted by $R(\Phi)$: $\Pi(x):=\argmin_{x'\in R(\Phi)}
\|x-x'\|_D^2$. The projection can be performed by the matrix
multiplication: we write $\Pi(x):=\Pi x$, where $\Pi:=\Phi (\Phi^T
D\Phi)^{-1}\Phi^T D$. Compared to TD~learning, the main advantage
of GTD~\cite{sutton2009convergent,sutton2009fast} algorithms are
their off-policy learning abilities.

\section{Distributed reinforcement learning overview}
Consider $N$ reinforcement learning agents labelled by $i \in \{
1,\ldots,N\}=:{\cal V}$. A multi-agent Markov decision process is
characterized by the tuple $(\{ {\cal S}_i \}_{i\in {\cal V}},\{
{\cal A}_i\}_{i\in {\cal V}},P,\{r_i\}_{i \in {\cal V}},\gamma)$,
where ${\cal S}_i$ is a finite state space (observations) of agent
$i$, ${\cal A}_i$ is a finite action space of agent $i$,
$r_i:{\cal S}_i\times {\cal A}_i\to {\mathbb R}$ is the reward
function, $\gamma \in (0,1)$ is the discount factor, and $P(\bar
s,\bar a,\bar s'):={\mathbb P}[\bar s'|\bar s,\bar a]$ represents
the unknown transition model of the joint state and action defined
as $\bar s:=(s_1,\ldots,s_N),\bar a:=(a_1,\ldots,a_N)$,
$\bar\pi(\bar s,\bar a):=\prod_{i=1}^N {\pi_i(s_i,a_i)}$, ${\cal
S}: = \prod\limits_{i = 1}^N {{\cal S}_i }$, ${\cal A}: =
\prod\limits_{i = 1}^N {{\cal A}_i }$. The stochastic policy of
agent $i$ is a mapping $\pi_i :{\cal S}_i \times {\cal A}_i\to
[0,1]$ representing the probability $\pi_i(s_i,a_i)={\mathbb
P}[a_i|s_i]$, $r_i^{\pi_i}:{\cal S}_i \to {\mathbb R}$ is defined
as $r_i^{\pi_i}(s_i):= {\mathbb E}_{a_i \sim \pi_i(s_i)}
[r_i(s_i,a_i)]$, $P^{\bar \pi}$ denotes the transition matrix,
whose $(\bar s,\bar s')$ entry is ${\mathbb P}[\bar s'|\bar
s]=\sum_{\bar a \in {\cal A}_1 \times \cdots \times {\cal A}_N}
{{\mathbb P}[\bar s'|\bar s,\bar a]\bar \pi(\bar s,\bar a)}$,
$d:{\cal S} \to {\mathbb R}$ denotes the stationary distribution
of the observation $\bar s \in {\cal S}$. We assume that each
agent can observe the entire joint states $\bar s$ and local
reward $r_i$. We consider the following assumption.
\begin{assumption}
With a fixed policy $\bar \pi$, the Markov chain $P^{\bar\pi}$ is
ergodic with the stationary distribution $d$ with $d(s) >0, s\in
{\cal S}$.
\end{assumption}
Throughout the paper, $D$ is defined as a diagonal matrix with
diagonal entries equal to those of $d$. The goal is to learn an
approximate value of the centralized reward
$r_c=(r_1^{\pi_1}+\cdots+r_N^{\pi_N})/N$.
\begin{problem}\label{problem:multi-agent-RL}
The goal of each agent $i$ is to learn an approximate value
function of the centralized reward $r_c=(r_1^{\pi_1}+\cdots+
r_N^{\pi_N})/N$ without knowledge of its transition model.
\end{problem}
\begin{remark}
Possible scenarios of~\cref{problem:multi-agent-RL} are summarized
as follows. Agents are located in a shared space, can observe the
joint states $\bar s$ from the environment, but get their own
local rewards. Another possibility is that each agent has its own
simulation environment and tries to learn the value of their
policy $\pi_i$ for the reward $r_i^{\pi_i}$. However, each agent
does not have access to other agents' rewards due to several
reasons. For instance, there exists no centralized coordinator;
thereby each agent does not know other agents' rewards. Another
possibility is that each agent/coordinator does not want to
uncover their own goal or the global goal for security/privacy
reasons.
\end{remark}
It can be proved that solving~\cref{problem:multi-agent-RL} is
equivalent to solving
\begin{align}
&\mathop{\min}_{w \in C}\sum_{i=1}^N {{\rm
MSPBE}_i(w)},\label{eq:distributed-opt0}
\end{align}
where $C \subset {\mathbb R}^q$ is a compact convex set which
includes the unique unconstrained global minimum
of~\eqref{eq:distributed-opt0}.
\begin{proposition}\label{prop:equivalance}
Solving~\eqref{eq:distributed-opt0} is equivalent to finding a
solution $w^*$ to the projected Bellman equation
\begin{align}
&\Pi\left( \frac{1}{N}\sum_{i=1}^N {r_i^{\pi_i} }+\gamma P^{\bar
\pi} \Phi w^* \right)=\Phi w^*.\label{eq:projected-Bellman-eq}
\end{align}
\end{proposition}
\begin{proof}
See~Appendix~\ref{appendix:proof-of-equivalence}.
\end{proof}
Equivalently, the problem can be written by the consensus
optimization~\cite{nedic2010constrained}
\begin{align}
&\mathop {\min}_{w_i\in C} \sum_{i=1}^N {{\rm
MSPBE}_i(w_i)}\label{eq:distributed-opt}\\
&{\rm subject\,\,to}\quad
w_1=w_2=\cdots=w_N.\label{eq:consensus-constraint}
\end{align}
To make the problem more feasible, we assume that its learning
parameter $w_i$ is exchanged via a communication network
represented by the undirected graph ${\cal G}=({\cal E},{\cal
V})$.

\section{Primal-dual distributed GTD algorithm (primal-dual DGTD)}
In this section, we study a distributed GTD algorithm. To this
end, we first define several vector and matrix notations to save
the space: $\bar w:= \begin{bmatrix}
   w_1\\
   \vdots\\
   w_N\\
\end{bmatrix}$, $\bar r^{\bar\pi}:=\begin{bmatrix}
   r_1^{\pi_1}\\
   \vdots\\
   r_N^{\pi _N}\\
\end{bmatrix}$, $\bar P^{\bar\pi}:=I_N \otimes P^{\bar\pi}$, $\bar L:=L \otimes
I_{|{\cal S}|}$, $\bar D:=I_N \otimes D$, $\bar\Phi:=I_N \otimes
\Phi$, and $\bar B:=\bar\Phi^T \bar D(I-\gamma\bar P^{\bar
\pi})\bar \Phi$. If we consider the loss function
in~\eqref{eq:GDD(0)-loss}, then the sum of loss functions
in~\eqref{eq:distributed-opt} can be compactly expressed as
$\sum_{i=1}^N {\rm MSPBE}_i(w_i)= \frac{1}{2}(\bar\Phi^T \bar
D\bar r^\pi-\bar B\bar w)^T (\bar \Phi^T \bar D\bar \Phi)^{-1}
(\bar\Phi^T \bar D\bar r^\pi-\bar B\bar w)$. Noting that the
consensus constraint~\eqref{eq:consensus-constraint} can be
expressed as
\begin{align*}
&\mathop{\min}_{\bar w} \frac{1}{2}(\bar\Phi^T \bar D\bar r^\pi-
\bar B\bar w)^T (\bar\Phi^T \bar D\bar \Phi )^{-1}(\bar\Phi^T \bar D\bar r^\pi-\bar B\bar w)\\
&{\rm subject\,\,to}\quad \bar L\bar w = 0
\end{align*}
and motivated
by~\cite{wang2010control,wang2011control,gharesifard2014distributed},
we convert it into the augmented Lagrangian
problem~\cite[sec.~4.2]{bertsekas1999nonlinear}
\begin{align}
&\mathop{\min}_{\bar w} \frac{1}{2}(\bar\Phi^T \bar D\bar r^\pi-\bar B\bar w)^T (\bar\Phi^T \bar D\bar\Phi)^{-1}(\bar\Phi^T \bar D\bar r^\pi-\bar B\bar w)\nonumber\\
&\quad\quad\quad +\bar w^T \bar L\bar L\bar w\label{eq:optimization1}\\
&{\rm subject\,\,to}\quad \bar L\bar w = 0.\nonumber
\end{align}
If the system is known, the above problem is an equality
constrained quadratic programming problem, which can be solved by
means of convex optimization methods~\cite{Boyd2004}. If the model
is unknown but observations can be sampled, then the problem can
be still solved by using stochastic optimization techniques. To
this end, some issues need to be carefully taken into account.
First, the objective function evaluation involves the double
sampling problem. Second, the inverse in the objective function
may lead to issues in developing algorithms. In
GTD~\cite{sutton2009fast}, this problem is resolved by a
decomposition technique. In~\cite{macua2015distributed}, it was
proved that the GTD can be related to the dual problem. Following
the same direction, we convert~\eqref{eq:optimization1} into the
equivalent optimization problem
\begin{align}
&\mathop{\min}_{\bar\varepsilon,\bar h,\bar w} \frac{1}{2}\bar
\varepsilon^T (\bar\Phi^T\bar D\bar\Phi)^{-1}\bar \varepsilon+\frac{1}{2}\bar h^T \bar h \label{eq:optimization2}\\
&{\rm subject\,\,to}\quad \begin{bmatrix}
   \bar B & I & 0\\
   \bar L & 0 & -I\\
   \bar L & 0 & 0\\
\end{bmatrix} \begin{bmatrix}
   \bar w\\
   \bar\varepsilon\\
   \bar h\\
\end{bmatrix}+ \begin{bmatrix}
   -\bar\Phi^T \bar D\bar r^{\bar\pi}\\
   0\\
   0\\
\end{bmatrix}=0,\nonumber
\end{align}
where $\bar\varepsilon,\bar h$ are newly introduced parameters.
Its Lagrangian dual can be derived by using standard
approaches~\cite{Boyd2004}.
\begin{proposition}\label{prop:dual-problem}
The Lagrangian dual problem of~\eqref{eq:optimization2} is given
by
\begin{align}
&\mathop{\min}_{\bar\theta,\bar v,\bar\mu}\psi (\bar\theta,\bar
v,\bar\mu)
\label{eq:dual-problem1}\\
&{\rm subject\,\,to}\quad \bar B^T \bar\theta-\bar L^T\bar v- \bar
L^T\bar\mu=0,\nonumber
\end{align}
where $\psi(\bar\theta,\bar v,\bar\mu ):= \frac{1}{2}\bar\theta^T
(\bar \Phi^T \bar D\bar\Phi)\bar \theta-\bar\theta^T\bar\Phi^T\bar
D\bar r^{\bar\pi}+\frac{1}{2}\bar v^T \bar v$.
\end{proposition}
\begin{proof}
The dual problem can be obtained by using the standard
manipulations in~\cite[Chap.~5]{Boyd2004}.
\end{proof}
As in~\cite{macua2015distributed}, we again construct the
Lagrangian function of~\eqref{eq:dual-problem1},
$L(\bar\theta,\bar v,\bar\mu,\bar w):=\psi(\bar\theta,\bar
v,\bar\mu)+[\bar B^T \bar\theta-\bar L^T\bar v-\bar L^T\bar\mu]^T
\bar w$, where $\bar w$ is the Lagrangian multiplier.
Since~\eqref{eq:dual-problem1} satisfies the Slater's
condition~\cite[pp.~226]{Boyd2004}, the strong duality holds,
i.e., $\mathop{\max}_{\bar w}\mathop{\min}_{\bar\theta,\bar v,\bar
\mu} L(\bar\theta,\bar v,\bar\mu,\bar
w)=\mathop{\min}_{\bar\theta,\bar v,\bar\mu}\mathop{\max}_{\bar w}
L(\bar\theta,\bar v,\bar\mu,\bar w)$, and the solutions
of~\eqref{eq:dual-problem1} are identical to solutions
$(\bar\theta^*,\bar v^*,\bar\mu^*,\bar w)$ of the saddle point
problem $L(\bar\theta^*,\bar v^*,\bar \mu^*,\bar w)\le
L(\bar\theta^*,\bar v^*,\bar\mu^*,\bar w^*) \le L(\bar\theta,\bar
v,\bar\mu,\bar w^*)$. In addition, the saddle points $(\bar
\theta^*,\bar v^*,\bar\mu^*,\bar w)$ satisfying the saddle point
problem are identical to the KKT points $(\bar\theta^*,\bar
v^*,\bar \mu^*,\bar w)$ satisfying
\begin{align}
&0 =\nabla_{\bar \theta} L(\bar\theta^*,\bar v^*,\bar\mu^*,\bar
w^*),\quad 0=\nabla_{\bar v}L(\bar\theta^*,\bar
v^*,\bar\mu^*,\bar w^*),\nonumber\\
&0=\nabla_{\bar\mu} L(\bar\theta^*,\bar v^*,\bar\mu^*,\bar
w^*),\quad 0 =\nabla_{\bar w} L(\bar\theta^*,\bar
v^*,\bar\mu^*,\bar w^*).\label{eq:KKT-points}
\end{align}
It is known in~\cite{wang2010control,wang2011control} that under a
certain set of assumptions the continuous gradient dynamics,
$\frac{d\bar\theta}{dt}=-\nabla_{\bar\theta} L(\bar\theta,\bar
v,\bar \mu,\bar w)$, $\frac{d\bar v}{dt} =-\nabla_{\bar
v}L(\bar\theta,\bar v,\bar\mu,\bar w)$, $\frac{d\bar
\mu}{dt}=-\nabla_{\bar\mu} L(\bar\theta,\bar v,\bar \mu,\bar w)$,
$\frac{d\bar w}{dt}=\nabla_{\bar w}L(\bar\theta ,\bar
v,\bar\mu,\bar w)$, of the Lagrangian function can solve the
saddle point problem. The dynamic systems can be compactly written
by the ODE $\dot x =-Ax-b$, where
\begin{align*}
&A:= \begin{bmatrix}
   \bar\Phi^T \bar D\bar \Phi & 0 & 0 & \bar \Phi ^T \bar D(I-\gamma\bar P^\pi)\bar \Phi\\
   0 & I & 0 & -\bar L\\
   0 & 0 & 0 & -\bar L\\
   -\bar\Phi^T (I-\gamma\bar P^\pi)^T \bar D\bar \Phi & \bar L & \bar L & 0\\
\end{bmatrix},\\
&b:= \begin{bmatrix}
   -\bar\Phi^T \bar D\bar r^{\bar\pi}\\
   0\\
   0\\
   0\\
\end{bmatrix},\quad x:= \begin{bmatrix}
   \bar\theta\\
   \bar v\\
   \bar\mu\\
   \bar w\\
\end{bmatrix}.
\end{align*}
We first establish the fact that the set of stationary points of
the ODE $\dot x = -Ax -b$ corresponds to the set of optimal
solutions of the consensus optimization
problem~\eqref{eq:consensus-constraint}.
\begin{proposition}\label{prop:stationary-points}
Consider the ODE $\dot x =-Ax-b$. The set of stationary points of
the ODE is given by ${\cal R}: = \{ \bar\theta^* \}\times \{\bar
v^*\}\times {\cal F} \times \{\bar w^*\}$, where $\bar v^*=0$,
$w^* = w_1^* = \cdots =w_N^*$, $w^*$ is the unique solution of the
projected Bellman equation~\eqref{eq:projected-Bellman-eq},
$\bar\theta^*= (\bar\Phi^T \bar D\bar\Phi)^{-1} \bar\Phi^T \bar
D(\bar r^{\bar\pi}- \bar \Phi\bar w^*+\gamma\bar P^{\bar\pi}\bar
\Phi\bar w^*)$, and ${\cal F}$ is the set of all solutions to the
linear equation for $\bar \mu$
\begin{align}
&{\cal F}:= \{\bar\mu:\bar L \bar \mu=\bar\Phi^T (I-\gamma\bar
P^\pi)^T \bar D\bar \Phi\bar\theta^*\}.\label{eq:set-F}
\end{align}
\end{proposition}
\begin{proof}
See Appendix~\ref{sec:stationary-point}.
\end{proof}
From~\cref{prop:stationary-points} and~\cref{prop:equivalance},
$w^*$ is the optimal solution of~\eqref{eq:distributed-opt0}. In
addition, the stationary points in~\cref{prop:stationary-points}
are the KKT points given in~\eqref{eq:KKT-points}. In addition, we
can prove that partial coordinates of the set of stationary points
in~\eqref{prop:stationary-points} are globally asymptotically
stable.
\begin{proposition}\label{prop:ODE-stability}
Consider the ODE $\dot x=-Ax-b$. Then, $(\bar \theta,\bar v,\bar
w)\to (\bar\theta^*,\bar v^*,\bar w^*)$ as $t\to\infty$.
\end{proposition}
\begin{proof}
See~Appendix~\ref{sec:appendox:ODE-stability}.
\end{proof}
Based on those observations, one can imagine a stochastic
approximation algorithm which can take benefits of the properties
of the ODE~$\dot x =-Ax-b$. In this respect, we propose the
distributed GTD (DGTD) in~\cref{algo:GDTD(0)}, where $C_{\bar
\theta},C_{\bar v},C_{\bar\mu},C_{\bar w}$ are box constraints
satisfying the following assumption.
\begin{assumption}
The constraint sets satisfy $\bar\theta^* \in C_{\bar\theta}$,
$\bar v^*\in C_{\bar v}$, $\bar w^*\in C_{\bar w}$, and $C_{\bar
\mu}\cap{\cal F}\neq\emptyset$.
\end{assumption}
The constraints are added to guarantee the stability and
convergence of the algorithm. According
to~\cite[Prop.~4]{bianchi2013convergence},\cite[Appendix~E]{bhatnagar2012stochastic},
the corresponding ODE is
\begin{align}
&\dot x =\Gamma_{T_C (x)}(-Ax-b),\label{eq:ODE1}
\end{align}
where $\Gamma_{T_C(x)}$ is defined as the projection of $x$ onto
the tangent cone $T_C(x)$~\cite[pp.~343]{bertsekas1999nonlinear}
of $C:=C_{\bar\theta}\times C_{\bar v} \times C_{\bar\mu}\times
C_{\bar w}$ at $x$. Due to the additional constraints, the set of
stationary points of~\eqref{eq:ODE1} is a larger set, which
includes those of $\dot x=-Ax-b$ as a subset. The following
results can be directly proved using the definitions of tangent
and normal cones~\cite[pp.~343]{bertsekas1999nonlinear}.
\begin{proposition}
The set of stationary points of~\eqref{eq:ODE1} is ${\cal P}:=\{
x\in C:\Gamma_{T_C(x)}(-Ax-b)=0\}=\{ x\in C:-Ax-b\in N_C(x)\}$.
\end{proposition}
We first establish the convergence of~\cref{algo:GDTD(0)} to the
stationary points of~\eqref{eq:ODE1} under the standard
diminishing step size rule~\cite{nedic2010constrained}
\begin{align}
&\alpha_k >0,\forall k \ge 0,\quad \sum_{k=0}^\infty {\alpha_k} =
\infty,\quad \sum_{k=0}^\infty
{\alpha_k^2}<\infty.\label{eq:diminishing-step-size-rule}
\end{align}
\begin{proposition}[Convergence of DGTD]\label{prop:convergence}
Define{\small
\begin{align*}
&\bar\theta_k:=\begin{bmatrix}
   \theta_{1,k}\\
   \vdots \\
   \theta_{N,k}\\
\end{bmatrix},\bar v_k := \begin{bmatrix}
   v_{1,k}\\
    \vdots\\
   v_{N,k}\\
\end{bmatrix},\bar\mu_k := \begin{bmatrix}
   \mu_{1,k}\\
    \vdots\\
   \mu_{N,k}\\
\end{bmatrix},\bar w_k:=\begin{bmatrix}
   w_{1,k}\\
    \vdots\\
   w_{N,k}\\
\end{bmatrix},
\end{align*}}
and $\bar x_k:= \begin{bmatrix}
   \bar\theta_k^T & \bar v_k^T & \bar\mu_k^T & \bar w_k^T\\
\end{bmatrix}^T$ with iterations in~\cref{algo:GDTD(0)}. With the step size
rule~\eqref{eq:diminishing-step-size-rule}, ${\rm dist}(\bar x_k
,{\cal P})\to 0$ as $k \to \infty$, where ${\cal P}:= \{x\in
C:-Ax-b\in N_C(x)\}$ with probability one.
\end{proposition}
\begin{proof}
See Appendix~\ref{appendix:proof-of-convergence}.
\end{proof}
\begin{algorithm}[h]
\caption{Distributed GTD algorithm (DGTD)}
\begin{algorithmic}[1]
\State Initialize $\{\theta_0^{(i)}\}_{i\in {\cal V}}$ and set
$k=0$.

\Repeat

\State $k \leftarrow k+1$

\For{agent $i\in \{1,\ldots,N\}$}

\State Sample $(s,a,s')$ with $s \sim d_i(s),a\sim \pi_i(a|s),s'
\sim p_i(s'|s,a)$ and update parameters according to
\begin{align*}
\theta_{i,k+1/2}=&\theta_{i,k}-
\alpha_k [\phi\phi^T \theta_{i,k}+\phi\phi^T w_{i,k}\\
&-\gamma\phi(\phi')^T w_{i,k}-\phi r_i^{\pi_i}]\\
v_{i,k+1/2} =&v_{i,k} - \alpha_k\left[v_{i,k}- \left( |{\cal
N}_i|w_{i,k} - \sum_{j\in
{\cal N}_i}w_{j,k} \right) \right],\\
\mu_{i,k+1/2}=& \mu_{i,k} + \alpha_k \left( |{\cal N}_i
|w_{i,k}-\sum_{j\in {\cal
N}_i}w_{j,k}\right),\\
w_{i,k+1/2}=& w_{i,k} - \alpha_k \left( |{\cal N}_i
|v_{i,k}-\sum_{j \in {\cal N}_i}v_{j,k}\right)\\
& - \alpha_k \left( |{\cal N}_i |\mu_{i,k}-\sum_{j\in {\cal
N}_i}\mu_{j,k} \right)\\
& +\alpha_k(\phi\phi^T \theta_{i,k}-\gamma\phi'\phi^T
\theta_{i,k}),
\end{align*}
where ${\cal N}_i$ is the neighborhood of node $i$ on the graph
$\cal G$,
$\phi:=\phi(s),\phi':=\phi(s'),r_i^{\pi_i}:=r_i^{\pi_i}(s)$.

\State Project the iterates
$\theta_{i,k+1}=\Gamma_{C_{\bar\theta}} [\theta_{i,k+1/2}]$,
$v_{i,k+1} =\Gamma_{C_{\bar v}} [v_{i,k+1/2}]$,
$\mu_{i,k+1}=\Gamma_{C_{\bar\mu}} [\mu_{i,k+1/2}]$,
$w_{i,k+1}=\Gamma_{C_{\bar w}}[ w_{i,k+1/2}]$.

\EndFor

\Until{a certain stopping criterion is satisfied.}

\end{algorithmic}
\label{algo:GDTD(0)}
\end{algorithm}
Although~\cref{prop:convergence} states that the iterations
of~\cref{prop:convergence} converge to a stationary point of the
projected ODE~\eqref{eq:ODE1}, it does not guarantee that they
converge to the set of stationary points of the ODE without the
projection in~\cref{prop:stationary-points}. In practice, however,
we expect that they may often converge to the set
in~\cref{prop:stationary-points}, if the constraint sets are
sufficiently large. On the other hand, if we follow the analysis
of the stochastic primal-dual algorithm
in~\cite{chen2016stochastic}, we can prove that under certain
conditions, the iterations of~\cref{algo:GDTD(0)} converge to the
the stationary points in~\cref{prop:stationary-points}. The proof
is similar to those in~\cite{chen2016stochastic}, and we defer its
full analysis to an extended version of this paper.

\section{Examples}
\begin{example}\label{ex:ex1}
Consider a stock market whose price process is approximated by a
Markov chain with 100 states ${\cal S}: = \{\$ 10,\$ 20,\ldots,\$
1000\}$. If an agent buys a stock, then it loses $s \in {\cal S}$,
and if sells, then it earns $s \in {\cal S}$. Define the trading
policy $\pi(s;a,b)=\begin{cases}
 {\rm If}\,\,a \le s \le b,\,{\rm then\,\,buy\,\,a\,stock} \\
 {\rm Otherwise,\,\,sell\,\,a\,\,stock}\\
 \end{cases}$. There are five trading agents ${\cal V} = \{1,2,\ldots, 5\}$ with
different private policies $\pi_1(s)=\pi(s;\$ 10,\$ 30)$,
$\pi_2(s)=\pi(s;\$ 10,\$ 40)$, $\pi_3(s)=\pi(s;\$ 10,\$ 50)$,
$\pi_4(s)=\pi(s;\$ 10,\$ 60)$, and $\pi_5(s)=\pi(s;\$ 10,\$ 70)$.
To determine an investment strategy, each agent is interested in
estimating an average of long term discounted profits of all
agents as well as its own. When the current state is $s\in {\cal
S}$, the reward of each agent is $r_i^{\pi_i}=-s$ if $\pi_i={\rm
buy}$, and $r_i^{\pi_i}=s$ if $\pi_i={\rm sell}$. For this
example, we used Gaussian radial basis functions as features of
the linear function approximation with 11 parameters, i.e.,
$w_i\in {\mathbb R}^{11}$, we considered the discount factor
$\gamma=0.5$, and we used a randomly generated Markov chain for
the stock price process model. Using the single agent
GTD~\cite{sutton2009fast}, each agent computed the approximate
value functions $J_{w_i^*} = \Phi w_i^*,i\in {\cal V}$. The
expected profits with the uniform initial state distribution are
${\mathbb E}_{s \sim U({\cal S})}[J_{w_1^*} (s)] = 164.3$,
${\mathbb E}_{s \sim U({\cal S})}[J_{w_2^*} (s)] =55.6$, ${\mathbb
E}_{s \sim U({\cal S})}[J_{w_3^*} (s)] =-107.5$, ${\mathbb E}_{s
\sim U({\cal S})}[J_{w_4^*} (s)] =-240.4$, and ${\mathbb E}_{s
\sim U({\cal S})}[J_{w_5^*} (s)] =-284.4639$, where $U({\cal S})$
is the uniform distribution over the state ${\cal S}$. Using the
single agent GTD again, the value function $J_{w_c^*}=\Phi w_c^*$
(global value function) corresponding to the central reward $r_c =
(r_1+r_2+r_3+r_4+r_5)/5$ was computed, and the expected profit is
${\mathbb E}_{s \sim U({\cal S})}[J_{w_c^*} (s)] =-82.5$. Since
each agent wants to keep its profit secure, agent $i$ can compute
its own value function $J_{w_i^*}$ only. However, there are
associated agents, which are able to exchange their parameters.
The associate relations are depicted in~\cref{fig:graph}.
\begin{figure}[h]
\centering\epsfig{figure=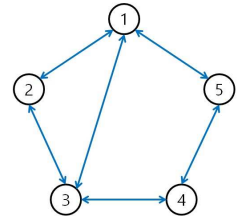,width=3cm} \caption{Graph
describing the associate relations among five trading
agents.}\label{fig:graph}
\end{figure}
Under this assumption,~\cref{algo:GDTD(0)} was applied with the
step size rule $\alpha_k=10/(k+1000)$ and without the projections,
and each agent computed the global value function estimations
$J_{\tilde w_i^*} = \Phi \tilde w_i^*,i\in {\cal V}$. The result
of 50000 iterations with a single simulation trajectory is
illustrated in~\cref{fig:convergence}. Distinguished by different
colors, the consensus of 11 parameters of $\tilde w_i$ for five
agents is shown. The same color is used for each coordinate of all
the agents. The expected profits with uniform initial state
distribution are ${\mathbb E}_{s \sim U({\cal S})}[J_{\tilde
w_1^*}(s)]=-83.2$, ${\mathbb E}_{s \sim U({\cal S})}[J_{\tilde
w_2^*}(s)]=-85.0$, ${\mathbb E}_{s \sim U({\cal S})}[J_{\tilde
w_3^*}(s)]=-81.2$, ${\mathbb E}_{s \sim U({\cal S})}[J_{\tilde
w_4^*}(s)]=-83.5$, and ${\mathbb E}_{s \sim U({\cal S})}[J_{\tilde
w_5^*}(s)]=-79.5$. This result demonstrates that each agent
successfully estimated the global value function $J_{w_c^*}=\Phi
w_c^*$.
\begin{figure}[h]
\centering\epsfig{figure=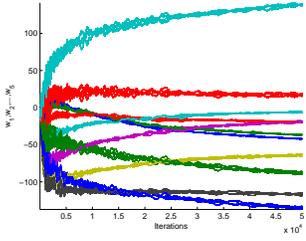,width=4.5cm}
\caption{\cref{ex:ex1}: Convergence of 11 parameters
(distinguished by different colors) of five agents (not
distinguished by colors).}\label{fig:convergence}
\end{figure}
Even though we are not aware of previous methods on the same
topic, we combined the standard consensus method with GTD and
compared the result with~\cref{fig:convergence}. We observed that
the convergence of~\cref{fig:convergence} is usually faster than
the standard consensus approach.
\end{example}
\begin{example}\label{ex:ex2}
Consider a multi-agent Markov decision process with 5 states
${\cal S}: = \{1,2,\ldots,5\}$ and
\begin{align*}
&P^{\bar \pi }  = \begin{bmatrix}
   0.2362 & 0.0895 & 0.3536 & 0.1099 & 0.2107  \\
   0.1821 & 0.2719 & 0.1553 & 0.1217 & 0.2689  \\
   0.1999 & 0.0279 & 0.2870 & 0.1628 & 0.3224  \\
   0.1149 & 0.1723 & 0.2726 & 0.3747 & 0.0656  \\
   0.2921 & 0.1719 & 0.0907 & 0.1836 & 0.2618  \\
\end{bmatrix},
\end{align*}
where the policy $\bar \pi$ is not specified. In addition,
consider 20 agents ${\cal V} = \{1,2,\ldots, 20\}$. They exchange
their parameters through a network described by a star graph,
where the agent~1 corresponds to the center node. The reward of
each agent~$i$ is $r_i^{\pi_i}\equiv i$ for all $i \in {\cal V}$.
As before, Gaussian radial basis functions are considered as
features of the linear function approximation with 3 parameters,
i.e., $w_i\in {\mathbb R}^{3}, i\in {\cal V}$. We
run~\cref{algo:GDTD(0)} with the discount factor $\gamma=0.5$ and
the step size rule $\alpha_k=2/(k+1000)$.
\begin{figure}[h]
\centering\epsfig{figure=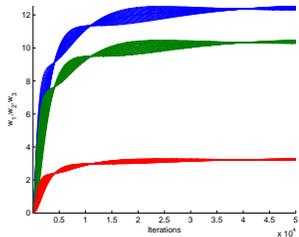,width=4.5cm}
\caption{\cref{ex:ex2}: Convergence of~3 parameters (distinguished
by different colors) of~20 agents.}\label{fig:convergence2}
\end{figure}
The result with 50000 iterations is illustrated
in~\cref{fig:convergence2}, where the consensus of 3 parameters in
$\tilde w_i$ is shown. The simulation result demonstrates the
validity of the proposed algorithm.
\end{example}

\section{Conclusion}
In this paper, we studied a new class of distributed GTD algorithm
based on primal-dual iterations, as compared
to~\cite{sutton2009fast}. The convergence was proved using
ODE-based methods. Simulation results demonstrated the
applicability of the proposed algorithm. Possible related future
research directions include extensions to actor-critic algorithms,
off-policy learning cases, and randomly changing networks.

\bibliographystyle{IEEEtran}
\bibliography{reference}

\begin{thebibliography}{10}
\providecommand{\url}[1]{#1}
\csname url@samestyle\endcsname
\providecommand{\newblock}{\relax}
\providecommand{\bibinfo}[2]{#2}
\providecommand{\BIBentrySTDinterwordspacing}{\spaceskip=0pt\relax}
\providecommand{\BIBentryALTinterwordstretchfactor}{4}
\providecommand{\BIBentryALTinterwordspacing}{\spaceskip=\fontdimen2\font plus
\BIBentryALTinterwordstretchfactor\fontdimen3\font minus
  \fontdimen4\font\relax}
\providecommand{\BIBforeignlanguage}[2]{{%
\expandafter\ifx\csname l@#1\endcsname\relax
\typeout{** WARNING: IEEEtran.bst: No hyphenation pattern has been}%
\typeout{** loaded for the language `#1'. Using the pattern for}%
\typeout{** the default language instead.}%
\else
\language=\csname l@#1\endcsname
\fi
#2}}
\providecommand{\BIBdecl}{\relax}
\BIBdecl

\bibitem{sutton2009convergent}
R.~S. Sutton, H.~R. Maei, and C.~Szepesv{\'a}ri, ``A convergent $o(n)$
  temporal-difference algorithm for off-policy learning with linear function
  approximation,'' in \emph{Advances in neural information processing systems},
  2009, pp. 1609--1616.

\bibitem{sutton2009fast}
R.~S. Sutton, H.~R. Maei, D.~Precup, S.~Bhatnagar, D.~Silver,
  C.~Szepesv{\'a}ri, and E.~Wiewiora, ``Fast gradient-descent methods for
  temporal-difference learning with linear function approximation,'' in
  \emph{Proceedings of the 26th Annual International Conference on Machine
  Learning}, 2009, pp. 993--1000.

\bibitem{kar2013cal}
S.~Kar, J.~M. Moura, and H.~V. Poor, ``{QD}-learning: a collaborative
  distributed strategy for multi-agent reinforcement learning through consensus
  $+$ innovations,'' \emph{IEEE Transactions on Signal Processing}, vol.~61,
  no.~7, pp. 1848--1862, 2013.

\bibitem{bertsekas1996neuro}
D.~P. Bertsekas and J.~N. Tsitsiklis, \emph{Neuro-dynamic programming}.\hskip
  1em plus 0.5em minus 0.4em\relax Athena Scientific Belmont, MA, 1996.

\bibitem{pennesi2010distributed}
P.~Pennesi and I.~C. Paschalidis, ``A distributed actor-critic algorithm and
  applications to mobile sensor network coordination problems,'' \emph{IEEE
  Transactions on Automatic Control}, vol.~55, no.~2, pp. 492--497, 2010.

\bibitem{zhang2018fully}
K.~Zhang, Z.~Yang, H.~Liu, T.~Zhang, and T.~Ba{\c{s}}ar, ``Fully decentralized
  multi-agent reinforcement learning with networked agents,'' \emph{arXiv
  preprint arXiv:1802.08757}, 2018.

\bibitem{macua2015distributed}
S.~V. Macua, J.~Chen, S.~Zazo, and A.~H. Sayed, ``Distributed policy evaluation
  under multiple behavior strategies,'' \emph{IEEE Transactions on Automatic
  Control}, vol.~60, no.~5, pp. 1260--1274, 2015.

\bibitem{stankovic2016multi}
M.~S. Stankovi{\'c} and S.~S. Stankovi{\'c}, ``Multi-agent temporal-difference
  learning with linear function approximation: weak convergence under
  time-varying network topologies,'' in \emph{American Control Conference
  (ACC)}, 2016, pp. 167--172.

\bibitem{mathkar2017distributed}
A.~Mathkar and V.~S. Borkar, ``Distributed reinforcement learning via gossip,''
  \emph{IEEE Transactions on Automatic Control}, vol.~62, no.~3, pp.
  1465--1470, 2017.

\bibitem{ram2010distributed}
S.~S. Ram, A.~Nedi{\'c}, and V.~V. Veeravalli, ``Distributed stochastic
  subgradient projection algorithms for convex optimization,'' \emph{Journal of
  optimization theory and applications}, vol. 147, no.~3, pp. 516--545, 2010.

\bibitem{bianchi2013convergence}
P.~Bianchi and J.~Jakubowicz, ``Convergence of a multi-agent projected
  stochastic gradient algorithm for non-convex optimization,'' \emph{IEEE
  Transactions on Automatic Control}, vol.~58, no.~2, pp. 391--405, 2013.

\bibitem{baird1995residual}
L.~Baird, ``Residual algorithms: Reinforcement learning with function
  approximation,'' in \emph{Machine Learning Proceedings 1995}, 1995, pp.
  30--37.

\bibitem{tutunov2016exact}
R.~Tutunov, H.~B. Ammar, and A.~Jadbabaie, ``An exact distributed newton method
  for reinforcement learning,'' in \emph{2016 IEEE 55th Conference on Decision
  and Control (CDC)}, 2016, pp. 1003--1008.

\bibitem{wang2010control}
J.~Wang and N.~Elia, ``Control approach to distributed optimization,'' in
  \emph{48th Annual Allerton Conference on Communication, Control, and
  Computing (Allerton)}, 2010, pp. 557--561.

\bibitem{wang2011control}
------, ``A control perspective for centralized and distributed convex
  optimization,'' in \emph{50th IEEE Conference on Decision and Control and
  European Control Conference (CDC-ECC)}, 2011, pp. 3800--3805.

\bibitem{gharesifard2014distributed}
B.~Gharesifard and J.~Cort{\'e}s, ``Distributed continuous-time convex
  optimization on weight-balanced digraphs,'' \emph{IEEE Transactions on
  Automatic Control}, vol.~59, no.~3, pp. 781--786, 2014.

\bibitem{chen2016stochastic}
Y.~Chen and M.~Wang, ``Stochastic primal-dual methods and sample complexity of
  reinforcement learning,'' \emph{arXiv preprint arXiv:1612.02516}, 2016.

\bibitem{mahadevan2014proximal}
S.~Mahadevan, B.~Liu, P.~Thomas, W.~Dabney, S.~Giguere, N.~Jacek, I.~Gemp, and
  J.~Liu, ``Proximal reinforcement learning: A new theory of sequential
  decision making in primal-dual spaces,'' \emph{arXiv preprint
  arXiv:1405.6757}, 2014.

\bibitem{borkar2000ode}
V.~S. Borkar and S.~P. Meyn, ``The {ODE} method for convergence of stochastic
  approximation and reinforcement learning,'' \emph{SIAM Journal on Control and
  Optimization}, vol.~38, no.~2, pp. 447--469, 2000.

\bibitem{bhatnagar2012stochastic}
S.~Bhatnagar, H.~Prasad, and L.~Prashanth, \emph{Stochastic recursive
  algorithms for optimization: simultaneous perturbation methods}.\hskip 1em
  plus 0.5em minus 0.4em\relax Springer, 2012, vol. 434.

\bibitem{kushner2003stochastic}
H.~Kushner and G.~G. Yin, \emph{Stochastic approximation and recursive
  algorithms and applications}.\hskip 1em plus 0.5em minus 0.4em\relax Springer
  Science \& Business Media, 2003, vol.~35.

\bibitem{sutton1998reinforcement}
R.~S. Sutton and A.~G. Barto, \emph{Reinforcement learning: {A}n
  introduction}.\hskip 1em plus 0.5em minus 0.4em\relax MIT Press, 1998.

\bibitem{nedic2010constrained}
A.~Nedic, A.~Ozdaglar, and P.~A. Parrilo, ``Constrained consensus and
  optimization in multi-agent networks,'' \emph{IEEE Transactions on Automatic
  Control}, vol.~55, no.~4, pp. 922--938, 2010.

\bibitem{bertsekas1999nonlinear}
D.~P. Bertsekas, \emph{Nonlinear programming}.\hskip 1em plus 0.5em minus
  0.4em\relax Athena scientific Belmont, 1999.

\bibitem{Boyd2004}
S.~Boyd and L.~Vandenberghe, \emph{Convex Optimization}.\hskip 1em plus 0.5em
  minus 0.4em\relax Cambridge University Press, 2004.

\end{thebibliography}

\appendix

\section{Proof of~\cref{prop:equivalance}}\label{appendix:proof-of-equivalence}
Since~\eqref{eq:distributed-opt0} is strongly convex, its
unconstrained global minimum is unique, and it satisfies
\begin{align*}
&\nabla_w \sum_{i=1}^N {{\rm MSPBE}_i(w)}=-(\Phi^T D(I-\gamma
P^{\bar\pi})\Phi)^T\\
&\quad \times (\Phi^T D\Phi)^{-1}\Phi^T D\sum_{i=1}^N
{(r_i^{\pi_i}-(I-\gamma P^{\bar\pi})\Phi w)}= 0.
\end{align*}
Since $\Phi^T D(I-\gamma P^{\bar\pi})\Phi$ is
nonsingular~\cite[pp.~300]{bertsekas1996neuro}, this implies
\begin{align*}
&(\Phi^T D\Phi)^{-1}\Phi^T D\sum_{i=1}^N{(r_i^{\pi_i}-(I-\gamma
P^{\bar\pi})\Phi w)}=0.
\end{align*}
Pre-multiplying the equation by $\Phi$ yields the desired result.

\section{Proof of~\cref{prop:convergence}}\label{appendix:proof-of-convergence}
The proof is based on the analysis of the stochastic recursion
\begin{align}
&x_{k+1}=\Gamma_C(f(x_k)+\varepsilon_k).\label{eq:appendix:eq8}
\end{align}
Define the $\sigma$-field ${\cal
F}_k:=\sigma(\varepsilon_0,\ldots,\varepsilon_{k-1},x_0,\ldots,x_{k},\alpha_0
,\ldots,\alpha_k)$. According
to~\cite[Prop.~4]{bianchi2013convergence},\cite[Appendix~E]{bhatnagar2012stochastic},
the corresponding ODE can be expressed as
\begin{align*}
&\dot x =\Gamma_{T_C(x)}[f(x)].
\end{align*}
We consider assumptions listed below.
\begin{assumption}\label{assumption:1}
$\,$\begin{enumerate}
\item The function $f:{\mathbb R}^N\to {\mathbb R}^N$ is
continuous.

\item The step sizes satisfy
\begin{align*}
&\alpha_k>0,\forall k \ge 0,\quad
\sum_{k=0}^\infty{\alpha_k}=\infty,\quad\alpha_k\to 0\,\,{\rm
as}\,\,k \to \infty .
\end{align*}

\item The ODE~$\dot x =\Gamma_{T_C(x)}[f(x)]$ has a compact subset
${\cal P}$ of ${\mathbb R}^N$ as its set of asymptotically stable
equilibrium points.
\end{enumerate}
\end{assumption}
Let $t(k),k\ge 0$ be a sequence of positive real numbers defined
according to $t(0)=0$ and $t(k)=\sum_{j=0}^{k-1}{\alpha_j},k \ge
1$. By the step size in~\cref{assumption:1}, $t(k) \to \infty$ as
$k\to \infty$. Define $m(t):=\max \{ k|t(k) \le t\}$. Thus,
$m(t)\to\infty$ as $t\to\infty$.
\begin{assumption}\label{assumption:2}
There exists $T$ such that for all $\delta>0$
\begin{align}
&\mathop{\lim}_{k\to\infty}{\mathbb P}\left( \mathop{\sup}_{j\ge
k}\mathop{\max}_{0\leq t\le T} \left\| \sum_{i=m(jT)}^{m(jT+t)-
1}{\alpha_i\varepsilon_i}\right\| \ge \delta \right)=0.
\label{eq:appendix:eq9}
\end{align}
\end{assumption}
\begin{lemma}[{Kushner and Clark Theorem~\cite[Appendix~E]{kushner2003stochastic}}]\label{lemma:Kushner}
Under~\cref{assumption:1} and~\cref{assumption:2}, for any initial
$x(0)\in {\mathbb R}^N$, $x(k)\to {\cal P}$ as $k\to\infty$ with
probability one.
\end{lemma}

{\bf Proof of~\cref{prop:convergence}}: We will check
\cref{assumption:1} and~\cref{assumption:2} and
use~\cref{lemma:Kushner} to complete the proof. The ODE
in~\eqref{eq:ODE1} is a projection of an affine map $f(x)=-Ax-b$;
therefore, it is obviously continuous. The step size assumption is
satisfied by the hypothesis. In addition, the set of stationary
points ${\cal P}=\{ x \in C:\Gamma_{T_C}(-Ax-b)=0\}$ is compact.
This is because ${\cal P}$ is expressed as ${\cal P} = \{x \in
C:-Ax-b \in N_C(x)\}$, where $N_C(x)$ is a convex closed cone, and
its pre-image of an affine map is also closed. Therefore, ${\cal
P}$ is closed. ${\cal P} \subseteq C$, because
$\Gamma_{T_C}(-Ax-b)=0$ only when $x\in C$. Since $C$ is compact
and ${\cal P}$ is its closed subset, ${\cal P}$ is also compact.
${\cal P}$ can be also proved to be globally asymptotically stable
following analysis given in~\cite{bianchi2013convergence}. For
completeness of the presentation, the brief proof is given in
Appendix~\ref{appendix:stable}. Next, we will
prove~\cref{assumption:2}. \cref{prop:convergence} can be
expressed as~\eqref{eq:appendix:eq8} with $\varepsilon_k=(-\tilde
Ax_k-\tilde b)-(-Ax_k-b)$, where $-\tilde Ax_k -\tilde b$ is a
stochastic approximate of $-Ax_k-b$ such that ${\mathbb E}[
-\tilde Ax_k -\tilde b |{\cal F}_k]=-Ax_k -b$. Therefore,
${\mathbb E}[\varepsilon_k|{\cal F}_k]=0$. Define
$M_k:=\sum_{i=0}^{k-1} {\alpha_i\varepsilon_i}$. Then, since
${\mathbb E}[M_{k+1}|{\cal F}_k]=M_k$, $(M_k)_{k=0}^\infty$ is a
Martingale sequence. We will prove the sufficient condition
for~\eqref{eq:appendix:eq9}.
\begin{align*}
&\mathop{\lim}_{k\to\infty}{\mathbb P}\left( \mathop{\max }_{0\le
t\le T} \left| \sum_{i=m(kT)}^{m(kT+t)-1} {\alpha_i\varepsilon_i}
\right|\ge \delta\right)=0.
\end{align*}
Since $\mathop{\max}_{0\le t\le T}
\left\|\sum_{i=m(kT)}^{m(kT+t)-1} {\alpha_i\varepsilon_i} \right\|
\le\mathop{\max}_{m(kT)\le t \le m(kT+T)-1} \left\|
\sum_{i=m(kT)}^t {\alpha_i\varepsilon_i}\right\|$, we will
consider a more conservative sufficient condition:
\begin{align}
&\mathop{\lim}_{k\to\infty}{\mathbb P}\left( \mathop{\max}_{0\le
t\le m(kT+T)-m(kT)-1} \|H_t\|\ge \delta
\right)=0,\label{eq:appendix:eq10}
\end{align}
where $(H_t)_{t=0}^\infty$ with $H_t:=\sum_{i=m(kT)}^{m(kT)+t}
{\alpha_i\varepsilon_i}$ is a Martingale sequence. Then, by using
the Martingale inequality, we have
\begin{align*}
&{\mathbb P}\left( \mathop{\max}_{0\le t\le
m(kT+T)-m(kT)-1}|H_t|\ge\delta\right)\\
&\le \frac{{\mathbb E}\left[ \left| \sum_{i=m(kT)}^{m(kT+T)-1}
{\alpha_i\varepsilon_i} \right|^2 \right]}{\delta^2}\le \frac{C^2
\sum_{i=m(kT)}^\infty {\alpha_i^2}}{\delta^2},
\end{align*}
where we used $\left\| \varepsilon_i \right\|^2 \le C^2$. By the
step size rule in~\eqref{eq:diminishing-step-size-rule},
$\sum_{k=0}^\infty{\alpha_k^2} <\infty$ implies that the
right-hand side converges to zero as $k \to \infty$. Therefore, we
prove~\eqref{eq:appendix:eq10} and~\eqref{eq:appendix:eq9}.
By~\cref{lemma:Kushner}, we prove that $x_k$ globally converges to
the stationary point ${\cal H}$ with probability one.

\section{Stationary point of~\eqref{eq:ODE1}}\label{appendix:stable}
In this section, we will prove the following claim.
\begin{proposition}
Consider the ODE $\dot x=\Gamma_{T_C(x)}(-Ax-b)$
in~\eqref{eq:ODE1}. The set of stationary points ${\cal H}=\{ x\in
C:\Gamma_{T_C}(-Ax-b)=0\}$ is globally asymptotically stable.
\end{proposition}
The proof is given in~\cite{bianchi2013convergence}, and we
provide a brief sketch of the proof.

\begin{proof}
If we define the function $V(x):=x^T(Ax+b)$, then the ODE can be
represented by $\dot x=\Gamma_{T_C(x)}(-\nabla_xV(x))$. Let $V(x)$
be a candidate Lyapunov function for ${\cal P}$. Then, its time
derivative is expressed as $\dot V(x) = \nabla_x
V(x)^T\Gamma_{T_C(x)}(-\nabla_x V(x))$. Since $\nabla_x
V(x)=\Gamma_{T_C(x)}(\nabla_x V(x))+\Gamma_{N_C(x)}(\nabla_x
V(x))$, and the tangent cone and normal cone are orthogonal, we
arrive at $\dot V(x)=-\left\| \Gamma_{T_C(x)}(-\nabla_x
V(x))\right\|^2$. Therefore, ${\cal P}=\{x\in
C:\Gamma_{T_C(x)}(-\nabla_x V(x))=0\}=\{x\in C:-\nabla_x V(x)\in
N_C(x)\}$ is globally asymptotically stable. Since $-\nabla_x V(x)
=-Ax-b$, the proof is completed.
\end{proof}

\section{Proof of~\cref{prop:stationary-points}}\label{sec:stationary-point}
We first consider the stationary points of~\eqref{eq:ODE1} without
the projection. They are obtained by solving the linear equation:
\begin{align}
&0=(\bar\Phi^T \bar D\bar\Phi)\bar\theta-\bar\Phi^T \bar D\bar
r^{\bar\pi}+\bar\Phi^T \bar D(I-\gamma \bar P^{\bar\pi})\bar \Phi \bar w,\label{eq:appendix:eq1}\\
&0=\bar v-\bar L\bar w,\label{eq:appendix:eq2}\\
&0=\bar L\bar w,\label{eq:appendix:eq4}\\
&0=\bar L\bar v +\bar L\bar \mu -\bar\Phi^T (I-\gamma\bar
P^{\bar\pi})^T \bar D\bar\Phi\bar\theta.\label{eq:appendix:eq3}
\end{align}
Since $\cal G$ is connected by~\cref{assumption:connected}, the
dimension of the null space of $L$ is one. Therefore, ${\rm
span}({\bf 1})$ is the null space.
Therefore,~\eqref{eq:appendix:eq4} implies the consensus
$w^*=w^*_1=\cdots=w^*_N$, and plugging~\eqref{eq:appendix:eq4}
into \eqref{eq:appendix:eq2} yields $\bar v^*=0$. With ${\bar
v}^*=0$,~\eqref{eq:appendix:eq3} is simplified to
\begin{align}
&\bar L\bar\mu^* = \bar\Phi^T(I-\gamma\bar P^{\bar\pi})^T \bar
D\bar\Phi \bar\theta^*.\label{eq:appendix:eq12}
\end{align}
In addition, from~\eqref{eq:appendix:eq1}, the stationary point
for $\bar \theta$ satisfies
\begin{align}
&\bar\theta^*=(\bar\Phi^T\bar D\bar\Phi )^{-1} \bar\Phi^T \bar
D(\bar r^{\bar\pi}-\bar\Phi\bar w^*+\gamma \bar
P^{\bar\pi}\bar\Phi\bar w^*).\label{eq:appendix:eq13}
\end{align}
Plugging the above equation into~\eqref{eq:appendix:eq12} yields
\begin{align}
\bar L \bar\mu &=\bar\Phi^T(I-\gamma\bar P^{\bar\pi})^T \bar
D\bar\Phi\bar\theta\nonumber\\
& =\bar\Phi^T (I-\gamma\bar P^{\bar\pi})^T \bar
D\bar\Phi(\bar\Phi^T\bar D\bar\Phi)^{-1}\nonumber\\
&\quad\quad\quad\quad\quad\quad \times \bar\Phi^T\bar D(\bar
r^{\bar\pi}-\bar\Phi\bar w + \gamma\bar P^{\bar\pi}\bar\Phi\bar
w)\label{eq:appendix:eq7}.
\end{align}
Multiplying \eqref{eq:appendix:eq7} by $({\bf 1}\otimes I)^T$ on
the left results in
\begin{align*}
&\sum_{i=1}^N ( [\Phi-\gamma\bar P^{\bar\pi}\Phi]^T
D\Phi(\Phi^T D\Phi)^{-1}\\
&\quad \Phi^T D[-r_i^{\pi_i}+\Phi w^*-\gamma\bar P^{\bar\pi}\Phi
w^*])=0,
\end{align*}
which is equivalent to $\sum_{i=1}^N{\nabla_w {\rm MSPBE}_i
(w^*)}=0$. Since the loss functions are strict convex quadratic
functions, $w^*$ is the unique global minimum of $\sum_{i=1}^N
{{\rm MSPBE}_i(w)}$. From~\eqref{eq:appendix:eq13}, $\bar\theta^*$
is also uniquely determined. In particular,
multiplying~\eqref{eq:appendix:eq1} by $({\bf 1} \otimes I)^T$
from the left, the unique stationary point for $\bar w^*$ is
expressed as $\bar w^*={\bf 1} \otimes w^*$ with
\begin{align*}
&w^*=\frac{1}{N}(\Phi^T D(I-\gamma\bar
P^{\bar\pi})\Phi)^{-1}\Phi^T D \left( \sum_{i=1}^N
{r_i^{\pi_i}}-N\Pi\right.\\
&\quad \left. \times\left(-\frac{1}{N}\sum_{i=1}^N
{r_i^{\pi_i}}+\Phi w_i^*-\gamma\bar P^{\bar\pi}\Phi w_i^*
\right)\right),
\end{align*}
From~\eqref{eq:appendix:eq7}, stationary ${\bar\mu}^*$ is any
solution of the linear equation~\eqref{eq:appendix:eq7}.

\section{Proof of~\cref{prop:ODE-stability}}\label{sec:appendox:ODE-stability}
Define
\begin{align*}
&x:=\begin{bmatrix}
   \bar\theta\\
   \bar v\\
\end{bmatrix},\quad y:=\bar\mu,\quad z:=\bar w,\\
&f(x,y):=\frac{1}{2}\bar\theta^T(\bar\Phi^T \bar D\bar
\Phi)\bar\theta-\bar\theta^T \bar\Phi^T \bar D\bar r^{\bar\pi}+ \frac{1}{2}\bar v^T \bar v,\\
&A:=\begin{bmatrix}
   \bar B^T & -\bar L^T & -\bar L^T\\
\end{bmatrix}.
\end{align*}
Then, the dual problem can be compactly expressed as $\min_{x,y}
f(x,y)\,\,{\rm s.t.}\,\,\,A\begin{bmatrix}
   x  \\
   y  \\
\end{bmatrix}=0$, and the ODE~\eqref{eq:ODE1} can be written by
\begin{align*}
&\begin{bmatrix}
   \dot x\\
   \dot y\\
\end{bmatrix}= - \begin{bmatrix}
   \nabla_x f(x,y)\\
   \nabla_y f(x,y)\\
\end{bmatrix}-A^T z,\dot z=A \begin{bmatrix}
   x\\
   y\\
\end{bmatrix}.
\end{align*}
The asymptotic stability analysis is based on the Lyapunov method
in the proof of~\cite[Thm.~2.1]{wang2011control}. However, the
proof in~\cite[Thm.~2.1]{wang2011control} cannot be directly
applied because $f(x,y)$ is not strictly convex in $y$, which
requires an additional analysis. Let $(x^*,y^*,z^*)$ be the
stationary point given in~\cref{sec:stationary-point}, and define
$(\tilde x,\tilde y,\tilde z):=(x-x^*,y-y^*,z-z^*)$. The
corresponding ODE is
\begin{align}
&\frac{d}{dt}\begin{bmatrix}
   \tilde x\\
   \tilde y\\
\end{bmatrix}=-\begin{bmatrix}
   \nabla_x f(x,y)\\
   \nabla_y f(x,y)\\
\end{bmatrix} + \begin{bmatrix}
   \nabla_x f(x^*,y^*)\\
   \nabla_y f(x^*,y^*)\\
\end{bmatrix} - A^T \tilde z,\nonumber\\
&\frac{d}{dt}\tilde z = A\begin{bmatrix}
   \tilde x\\
   \tilde y\\
\end{bmatrix}.\label{eq:appendix:eq6}
\end{align}
Consider the quadratic candidate Lyapunov function
\begin{align*}
&V(\tilde x,\tilde y,\tilde z): = \frac{1}{2}
\begin{bmatrix}
  \tilde x\\
  \tilde y\\
\end{bmatrix}^T \begin{bmatrix}
  \tilde x\\
  \tilde y\\
\end{bmatrix} + \frac{1}{2}\tilde z^T \tilde z,
\end{align*}
whose time derivative is
\begin{align*}
&\frac{d}{dt}V(\tilde x,\tilde y,\tilde z) =- \begin{bmatrix}
   \tilde x\\
   \tilde y\\
\end{bmatrix}^T \begin{bmatrix}
   \nabla_x f(x,y)\\
   \nabla_y f(x,y)\\
\end{bmatrix} + \begin{bmatrix}
   \tilde x\\
   \tilde y\\
\end{bmatrix}^T \begin{bmatrix}
   \nabla_x f(x^*,y^*)\\
   \nabla_y f(x^*,y^*)\\
\end{bmatrix}.
\end{align*}
Since $f$ is convex, the gradient satisfies the global
under-estimator property
\begin{align*}
f(x',y')\ge& f(x,y)\\
& + \begin{bmatrix}
   \nabla_x f(x,y) \\
   \nabla_y f(x,y) \\
\end{bmatrix}^T \left(\begin{bmatrix}
   x'\\
   y'\\
\end{bmatrix}-\begin{bmatrix}
   x\\
   y\\
\end{bmatrix} \right),\forall \begin{bmatrix}
   x'\\
   y'\\
\end{bmatrix}, \begin{bmatrix}
   x\\
   y\\
\end{bmatrix}.
\end{align*}
Since $f(x,y)$ only depends on $x$ and is strictly convex in $x$,
a strict inequality holds if and only if $x'=x$. Therefore, the
following holds:
\begin{align*}
&f(x,y)>f(x^*,y^*)+\begin{bmatrix}
   \nabla_x f(x^*,y^*)\\
   \nabla_y f(x^*,y^*)\\
\end{bmatrix}^T \begin{bmatrix}
   \tilde x\\
   \tilde y\\
\end{bmatrix},\\
&f(x^*,y^*)>f(x,y)- \begin{bmatrix}
   \nabla_x f(x,y)\\
   \nabla_y f(x,y)\\
\end{bmatrix}^T \begin{bmatrix}
   \tilde x\\
   \tilde y\\
\end{bmatrix},
\end{align*}
if and only if $\tilde x \ne 0$. Adding both sides of the
inequalities leads to $\frac{d}{dt}V(\tilde x,\tilde y,\tilde z) <
0$ if and only if $\tilde x \ne 0$. Therefore, one concludes that
$\bar\theta \to \bar\theta^*$ and $\bar v\to \bar v^*=0$. Since
$\frac{d}{dt}V(\tilde x,\tilde y,\tilde z)=0,\forall (\tilde
x,\tilde y,\tilde z)\in {\cal G}:=\{\tilde x,\tilde y,\tilde
z:\tilde x=0\}$, we invoke LaSalle invariant principle to prove
that all bounded trajectories converge to the largest invariant
set $\cal M$ such that ${\cal M} \subseteq {\cal G}$. Now, we
focus on the trajectories $(\tilde x(t),\tilde y(t),\tilde z(t))
\in {\cal G}$, where the ODE~\eqref{eq:appendix:eq6} becomes
$\frac{d}{dt}\begin{bmatrix}
   0  \\
   \tilde y\\
\end{bmatrix}= -A^T \tilde z,\frac{d}{dt}\tilde z = A \begin{bmatrix}
   0\\
   \tilde y\\
\end{bmatrix}$. In particular, we have $0=\bar B(\bar w - \bar w^*),0=\bar L(\bar w-\bar w^*),\frac{d}{dt}\bar\mu=\bar L(\bar w-\bar w^*),\frac{d}{dt}\bar w=-\bar L^T (\bar\mu-\bar\mu^*)$.
Since $\bar B:= \bar\Phi^T \bar D(I-\gamma\bar
P^{\bar\pi})\bar\Phi $ is nonsingular, we have $\bar w=\bar w^*$,
$\frac{d}{dt}\bar\mu=0$, and $0=-\bar L^T (\bar\mu-\bar \mu^*)$.
This implies
\begin{align*}
&{\cal M} = \left\{ \begin{bmatrix}
   \bar \theta\\
   \bar v\\
   \bar \mu\\
   \bar w\\
\end{bmatrix}:\begin{bmatrix}
   \bar \theta\\
   \bar v\\
   \bar \mu\\
   \bar w\\
\end{bmatrix} = \begin{bmatrix}
   \bar\theta^*\\
   \bar v^*\\
   \bar \mu^*\\
   \bar w^*\\
\end{bmatrix},\bar\mu^*\in {\cal F} \right\},
\end{align*}
where ${\cal F}$ is defined in~\eqref{eq:set-F}. Therefore, all
bounded solutions converge to ${\cal M}$. However, the boundedness
of the solutions is not guaranteed. By Lyapunov inequality
$\frac{d}{dt}V(\tilde x,\tilde y,\tilde z)\leq 0,\forall (\tilde
x,\tilde y,\tilde z)$, trajectory $(\bar\theta,\bar v,\bar w)$ is
guaranteed to be bounded, while $\bar\mu$ may not because the set
of stationary points ${\cal F}$ of $\bar\mu$ defined
in~\eqref{eq:set-F} is an unbounded affine space. However,
LaSalle's invariance principle can be applied to those bounded
partial coordinates. Therefore, we have $(\bar\theta,\bar v,\bar
w)\to (\bar\theta^*,\bar v^*,\bar w^*)$ as $t \to \infty$
globally. This completes the proof.
\end{document}